\documentclass[11pt]{article}
\usepackage{amssymb,amsmath,amsthm,verbatim}
\usepackage{fullpage}
\usepackage{xcolor}
\usepackage{url}
\usepackage{mathrsfs}
\usepackage[all]{xy}
  \SelectTips{cm}{10}
  \everyxy={<2.5em,0em>:}
\usepackage{tikz}
\usepackage{picinpar} 

\usepackage{fancyhdr}
\usepackage{ifpdf}
  \ifpdf
    \usepackage{hyperref} 
  \else
    
    \newcommand{\href}[2]{#2}
  \fi

\theoremstyle{plain}
  \newtheorem{lemma}[equation]{Lemma}
  
  \newtheorem{theorem}[equation]{Theorem}
  \newtheorem{corollary}[equation]{Corollary}
    \newtheorem{question}{Question}

\theoremstyle{definition}

\theoremstyle{remark}
  \newtheorem{remark}[equation]{Remark}

 \DeclareFontFamily{U}{manual}{}
 \DeclareFontShape{U}{manual}{m}{n}{ <->  manfnt }{}
 \newcommand{\manfntsymbol}[1]{%
    {\fontencoding{U}\fontfamily{manual}\selectfont\symbol{#1}}}

\makeatletter
   \@addtoreset{section}{part}
   \@addtoreset{equation}{section}
   \@addtoreset{footnote}{section}

    {\hspace*{\fill}$\lrcorner$\endgraf\endgroup\end{trivlist}}
 \newenvironment{example}[1][]{
   \refstepcounter{equation}
   \begin{proof}[Example~\theequation%
   \@ifnotempty{#1}{ (#1)}.]
   }
  {\end{proof}}
\makeatother

  \DeclareFontFamily{OT1}{pzc}{}
  \DeclareFontShape{OT1}{pzc}{m}{it}{<-> s * [1.100] pzcmi7t}{}
  \DeclareMathAlphabet{\mathpzc}{OT1}{pzc}{m}{it}

\newif\ifhascomments \hascommentstrue
\ifhascomments
  \newcommand{\anton}[1]{{\color{red}[[\ensuremath{\bigstar\bigstar\bigstar} #1]]}}
  \newcommand{\matt}[1]{{\color{red}[[\ensuremath{\spadesuit\spadesuit\spadesuit} #1]]}}
\else
  \newcommand{\anton}[1]{}
  \newcommand{\matt}[1]{}
\fi

\renewcommand{\AA}{\mathbb{A}}

\DeclareMathOperator{\Aut}{\ensuremath{\mathcal{A}\kern-.125em\mathpzc{ut}}}

\newcommand{\can}{\mathrm{can}}

\newcommand{\C}{\mathcal C}
\newcommand{\CC}{\mathbb C}

\newcommand{\D}{\mathcal D}

\DeclareMathOperator{\Endo}{\ensuremath{\mathcal{E}\kern-.125em\mathpzc{nd}}}

\newcommand{\GG}{\mathbb G}

\DeclareMathOperator{\Hom}{\ensuremath{\mathcal{H}\kern-.125em\mathpzc{om}}}

\renewcommand{\L}{\mathcal L}

\renewcommand{\O}{\mathcal O}
\DeclareMathOperator{\ord}{ord}

\newcommand{\RR}{\mathbb R}

\DeclareMathOperator{\spec}{Spec}

\newcommand{\ttilde}[1]{\widetilde{#1}}

\newcommand{\U}{\mathcal U}

\newcommand{\X}{\mathcal{X}}
\newcommand{\Y}{\mathcal{Y}}
\newcommand{\Z}{\mathcal{Z}}
\newcommand{\ZZ}{\mathbb{Z}}

 \def\ari[#1]{\ar@{^(->}[#1]}
 \def\are[#1]{\ar[#1]^{\txt{\'et}}}
 \def\areh[#1]{\ar[#1]|{\txt{$H$-eq}}^{\txt{\'et}}}
 \def\ars[#1]{\ar@{->>}[#1]}
 \newcommand{\dplus}{\ar@{}[d]|{\mbox{$\oplus$}}}
 \newcommand{\dtimes}{\ar@{}[d]|{\mbox{$\times$}}}

\begin{document}
\title{A ``bottom up'' characterization of smooth Deligne-Mumford stacks}
\author{Anton Geraschenko\thanks{\texttt{geraschenko@gmail.com}}\ \ and Matthew Satriano\thanks{Supported by NSF grant DMS-0943832 and an NSF postdoctoral fellowship (DMS-1103788).}}
\date{}
\maketitle

In casual discussion, a stack $\X$ is often described as a variety $X$ (the coarse space of $\X$) together with stabilizer groups attached to some of its subvarieties. However, this description does not uniquely specify $\X$, as is illustrated by Example \ref{eg:non-affine-lines}. Our main result shows that for a large class of stacks one typically encounters, this description does indeed characterize them. Moreover, we prove that each such stack can be described in terms of two simple procedures applied iteratively to its coarse space: canonical stack constructions and root stack constructions.

More precisely, if $\X$ is a smooth separated tame Deligne-Mumford stack of finite type over a field $k$ with trivial generic stabilizer, it is completely determined by its coarse space $X$ and the ramification divisor (on $X$) of the coarse space morphism $\X\to X$. Therefore, to specify such a stack, it is enough to specify a variety and the \emph{orders} of the stabilizers of codimension 1 points. The group structures, as well as the stabilizer groups of higher codimension points, are then determined.

Recall that a $k$-scheme $U$ is said to have \emph{(tame) quotient singularities} if it is \'etale locally a quotient of a smooth variety by a finite group (of order relatively prime to the characteristic of $k$). The coarse space of a smooth tame Deligne-Mumford stack has tame quotient singularities. We say that an algebraic stack $\Y$ has (tame) quotient singularities if there is an \'etale cover $U\to\Y$, where $U$ is a scheme with (tame) quotient singularities. As explained in the next section, we can associate to every such stack a canonical smooth stack $\Y^\can$ over $\Y$. Also discussed in the next section, for an effective Cartier divisor $\C$ on a stack $\Z$ and a positive integer $n$, there is a universal morphism to $\Z$ which ramifies to order $n$ over $\C$, called the \emph{$n$-th root of $\C$}, denoted $\sqrt[n]{\C/\Z}$ (or $\sqrt{\C/\Z}$ if there is no confusion about what $n$ is).

\begin{theorem}\label{thm:bottom-up-smooth-DM}
  Let $\X$ be a smooth separated tame Deligne-Mumford stack of finite type over a field $k$ with trivial generic stabilizer. Let $X$ denote the coarse space of $\X$, $D\subseteq X$ the ramification divisor of the coarse space map $\pi\colon \X\to X$, and $\D\subseteq X^\can$ the corresponding divisor in $X^\can$. Let $e_i$ be the ramification degrees of $\pi$ over the irreducible components $\D_i$ of $\D$, and let $\sqrt{\D/X^\can}$ denote the root stack of $X^\can$ with order $e_i$ along $\D_i$.
  Then $\sqrt{\D/X^\can}$ has tame quotient singularities and $\pi$ factors as follows:
  \[
   \X \cong \sqrt{\D/X^\can}^\can \to \sqrt{\D/X^\can} \to X^\can \to X.
  \]
Moreover, if $D$ is Cartier (so that $\sqrt{D/X}$ is defined), then $\sqrt{D/X}$ has tame quotient singularities and $\pi$ factors as follows:
  \[
   \X\cong \sqrt{D/X}^\can\to \sqrt{D/X}\to X.
  \]
\end{theorem}

\begin{remark}
  For $\X$ with non-trivial generic stabilizer, Theorem \ref{thm:bottom-up-smooth-DM} can be combined with \cite[Theorem A.1]{AOV}, which shows that $\X$ is a gerbe over a smooth stack with trivial generic stabilizer.
\end{remark}
\begin{remark}
\label{doesnt-matter-how-you-root}
In the statement of Theorem \ref{thm:bottom-up-smooth-DM}, there are several ways one can construct $\sqrt{\D/X^\can}$. For example, it can be constructed as an iterated sequence of root stacks along the irreducible components
$\sqrt[e_m]{\D_m/\dots\sqrt[e_1]{\D_1/X^\can}}$. Alternatively, if say $e_1=e_2$, one can combine the two steps of rooting along $\D_1$ and $\D_2$ into a single step $\sqrt[e_1]{(\D_1\cup\D_2)/X^\can}$. As noted in the last paragraph of \cite[\S1.3b]{fmn}, these two stacks are \emph{not} isomorphic. However, the proof of Theorem \ref{thm:bottom-up-smooth-DM} shows that both have tame quotient singularities and their canonical stacks are isomorphic.
\end{remark}
\begin{remark}
  The proof we present can be applied if $k$ is an arbitrary regular excellent base, but we work over a field for clarity. See \cite{Vistoli-canonical} for the construction of canonical stacks in this case. The remainder of the proof applies directly.
\end{remark}

\subsection*{Acknowledgments}
It is a pleasure to thank Stephen Griffeth, Christian Liedtke, and Angelo Vistoli for their help with several technical points \cite{griffeth, Vistoli-canonical}.

\section*{Background and preliminary results}

The first stack-theoretic construction we employ is that of canonical stacks. If $X$ is a finite type algebraic space, then \cite[2.9 and the proof of 2.8]{vistoli:intersection} shows $X$ has tame quotient singularities if and only if $X$ is the coarse space of a smooth tame Deligne-Mumford stack. Moreover, in this case there is a smooth tame Deligne-Mumford stack $X^\can$ with coarse space $X$ such that the coarse space morphism $X^\can\to X$ is an isomorphism away from codimension 2. We refer to $X^\can$ as \emph{the canonical stack} over $X$.
\begin{remark}[Universal property of canonical stacks {\cite[Theorem 4.6]{fmn}}]\label{rmk:canonical-universal-property}
  Suppose $\X$ is a smooth Deligne-Mumford stack with coarse space morphism $\pi\colon\X\to X$ which is an isomorphism away from a locus of codimension at least 2. Then $\X$ is universal (terminal) among smooth Deligne-Mumford stacks with trivial generic stabilizer and a dominant codimension-preserving morphism to $X$.
\end{remark}
This universal property is stable under base change by \'etale morphisms (or any other codimension-preserving morphisms). That is, for $U\to X$ \'etale, $X^\can\times_X U$ is the canonical stack over $U$. By descent, every algebraic stack $\Y$ with tame quotient singularities has a canonical stack $\Y^\can$ and a coarse space morphism $\pi:\Y^\can \to\Y$ which is an isomorphism away from a set of codimension at least 2.

%

The second stack-theoretic construction we employ is that of root stacks (see \cite[\S1.3]{fmn} for a brief introduction). Given an effective Cartier divisor $\C$ on a stack $\Z$ and a positive integer $n$, the \emph{$n$-th root of $\C$}, denoted $\sqrt[n]{\C/\Z}$, is the $\Z$-stack which is universal (terminal) among all $\Z$-stacks for which $\C$ pulls back to $n$ times another divisor. Explicitly, the root stack is given by the cartesian diagram below. Recall that $\C$ defines a morphism $\Z\to [\AA^1/\GG_m]$ since $[\AA^1/\GG_m]$ parameterizes line bundles with sections (i.e.~effective Cartier divisors).
\[\xymatrix{
  \sqrt[n]{\C/\Z}\ar[r]\ar[d] & [\AA^1/\GG_m]\ar[d]^{\hat{\;}n}\\
  \Z\ar[r]^-{\C} & [\AA^1/\GG_m].
}\]
Here, the $n$-th power map $\hat{\;}n\colon [\AA^1/\GG_m]\to [\AA^1/\GG_m]$ takes a line bundle with section $(\L,s)$ to $(\L^{\otimes n},s^{\otimes n})$.

\begin{lemma}\label{lem:affine-diag}
Root stack morphisms and canonical stack morphisms have affine diagonal.
\end{lemma}
\begin{proof}
To check that root stack morphisms have affine diagonal, it is enough to show that the $n$-th power map $[\AA^1/\GG_m]\to [\AA^1/\GG_m]$ has affine diagonal. Since this can be checked smooth locally on the target, and $[\AA^1/\mu_n] = \AA^1\times_{[\AA^1/\GG_m]}[\AA^1/\GG_m] $, the result follows.

We next consider canonical stack morphisms. Note that if $H$ is a finite group acting on an affine variety $U$, then the coarse space map $[U/H]\to U/H$ has affine diagonal. Since all canonical stack morphisms have this form smooth locally on the target, they have affine diagonal as well.
\end{proof}

\begin{lemma}\label{lem:rep-local}
 Suppose $f\colon \X\to \Y$ is a morphism of algebraic stacks, and $\Z\to \Y$ is a surjective locally finite type morphism. Then $f$ is representable if and only if $f_\Z:\X\times_\Y \Z\to \Z$ is representable.
\end{lemma}
\begin{proof}
  By \cite[Corollary 2.2.7]{conrad:elliptic}, a morphism is representable if and only if its geometric fibers are algebraic spaces. Since $\Z\to \Y$ is locally finite type, the geometric fibers of $f$ and $f_\Z$ are the same.
\end{proof}

\begin{lemma}
\label{l:etovercs}
  Let $\U$ be a separated Deligne-Mumford stack with trivial generic stabilizer.  If $\U$ is \'etale over its coarse space, then $\U$ is an algebraic space.
\end{lemma}
\begin{proof}
  It suffices to look \'etale locally on the coarse space of $\U$.  We can therefore assume $\U=[V/K]$, where $V$ is an algebraic space and $K$ is the stabilizer of a geometric point $v$ of $V$ by \cite[Lemma 2.2.3 and its proof\footnote{In the proof of \cite[Lemma 2.2.3]{Abramovich-Vistoli}, the group $\Gamma$ is the stabilizer of the $x_0$.}]{Abramovich-Vistoli}. Then the composition $V\to [V/K]\to V/K$ is \'etale, so $K$ acts trivially on every jet space of $v$ by the formal criterion for \'etaleness. On the other hand, $K$ acts faithfully on some jet space of $v$ by the proof of \cite[Proposition 4.4]{brauer} (smoothness is not needed). Thus, $K$ is trivial, so $\U=V$ is an algebraic space.
\end{proof}

\begin{corollary}
  \label{cor:etbirat==>rep}
  Let $\U$ and $\Y$ be smooth separated tame Deligne-Mumford stacks.  Suppose a morphism $g\colon\U\to\Y$ is birational, and the ramification locus in $\U$ is of codimension greater than 1. Then $g$ is representable and \'etale.
\end{corollary}
\begin{proof}
  By Lemma \ref{lem:rep-local}, we can replace $\Y$ by an \'etale cover, and so can assume $\Y=Y$ is a scheme. Then $\U$ has trivial generic stabilizer and $g$ factors through the coarse space $U$ of $\U$.

  We show that the induced map $\bar g\colon U\to Y$ is unramified in codimension 1.  Let $u\in U$ be a point of codimension 1 and let $y\in Y$ be its image in $Y$.  Being the coarse space of a smooth tame Deligne-Mumford stack, $U$ has tame quotient singularities. In particular, it is normal, and so $\O_{U,u}$ is a discrete valuation ring.  If $V\to\U$ is an \'etale cover and $v\in V$ is a point of codimension 1 which maps to $u$, then $\O_{V,v}$ is unramified over $\O_{Y,y}$ by hypothesis.  Since ramification indices multiply in towers of discrete valuation rings, we see that $\O_{U,u}$ is unramified over $\O_{Y,y}$ as well.

  Now we have that $U$ is normal, $Y$ is smooth, and $\bar g$ is dominant and unramified in codimension 1, so the purity of the branch locus theorem \cite[Expos\'e X, Theorem 3.1]{sga1} shows that $\bar g$ is \'etale. Similarly, $V\to Y$ is a map between two smooth varieties which is dominant and unramified in codimension 1, so it is \'etale by purity. It follows that $g$ is \'etale. Since $g$ and $\bar{g}$ are \'etale, the coarse space map $\U\to U$ is \'etale, and hence an isomorphism by Lemma \ref{l:etovercs}. This proves representability of $g$.
\end{proof}

\begin{corollary}\label{lem:unram-iso-on-coarse-implies-iso}
  Let $f\colon \X\to \Y$ be a morphism of smooth separated tame Deligne-Mumford stacks with trivial generic stabilizers. If $f$ is unramified in codimension 1 and induces an isomorphism of coarse spaces, then $f$ is an isomorphism.
\end{corollary}
\begin{proof}
  Let $\pi\colon \Y\to Y$ be the coarse space of $\Y$. By assumption, the composition $\pi\circ f\colon \X\to Y$ is a coarse space morphism. Since $\X$ and $\Y$ have trivial generic stabilizers, these coarse space morphisms are birational, so $f$ is birational. Since $\X$ and $\Y$ are proper and quasi-finite over $Y$ \cite[Theorem 1.1]{conrad-KM}, $f$ is proper and quasi-finite. Since $f$ is unramified in codimension 1, it is representable by Corollary \ref{cor:etbirat==>rep}.  Zariski's Main Theorem \cite[Theorem C.1]{fmn} then shows that $f$ is an isomorphism.
\end{proof}

\section*{Proof of Theorem \ref{thm:bottom-up-smooth-DM}}
  By the universal property of canonical stacks (Remark \ref{rmk:canonical-universal-property}), the coarse space morphism $\pi\colon \X\to X$ factors through the canonical stack. Since $X^\can\to X$ is an isomorphism away from codimension 2, we have that $\D$ is the ramification divisor of the map $\ttilde\pi\colon \X\to X^\can$.

  Let $\D_i$ be the irreducible components of $\D$ and suppose $\ttilde\pi$ is ramified over $\D_i$ with order $e_i$. Then $\Y=\sqrt{\D/X^\can}$ is obtained from $X^\can$ by taking the $e_i$-th root along $\D_i$ for all $i$. By the universal property of root stacks, we have an induced morphism $g:\X\to\Y$.

  Note that each $e_i$ is relatively prime to the characteristic of $k$; indeed, this can be checked \'etale locally on $X$, so we may assume $\X$ is a quotient of a smooth scheme by a finite group $G$ which stabilizes a point \cite[Lemma 2.2.3 and its proof]{Abramovich-Vistoli}. The ramification orders are then orders of subgroups of $G$, which are relatively prime to the characteristic of $k$ since $\X$ is assumed to be tame. Since the $e_i$ are prime to the characteristic of $k$, $\Y$ is a tame Deligne-Mumford stack. Since $\Y\to X^\can$ ramifies with order $e_i$ over $\D_i$ and ramification orders multiply in towers of discrete valuation rings, we have that $g$ is unramified in codimension 1.

  Let $\U\subseteq X^\can$ be the complement of the singular loci of the $\D_i$ and the intersections of the $\D_i$. The restriction of $\D$ to $\U$ is a simple normal crossing divisor, so the open substack  $\Y\times_{X^\can}\U$ of $\Y$ is a smooth Deligne-Mumford stack by \cite[1.3.b(3)]{fmn}. By Corollary \ref{lem:unram-iso-on-coarse-implies-iso}, $g$ restricts to an isomorphism of open substacks $\X\times_{X^\can}\U\to \Y\times_{X^\can}\U$. These open substacks have complements of codimension at least 2, so $g$ is Stein (i.e.~$g_*\O_\X\cong \O_\Y$).

  By Lemma \ref{lem:affine-diag}, $\Y\to X$ has affine diagonal. As $\pi\colon \X\to X$ is a coarse space morphism of a tame stack, it is cohomologically affine \cite[Theorem 3.2]{AOV}. Since $\X$ is separated, 
  $X$ is separated \cite[Theorem 1.1]{conrad-KM}, and so $X$ has quasi-affine diagonal. By \cite[Proposition 3.14]{Alper:good}, $g\colon \X\to \Y$ is cohomologically affine. As $g$ is Stein and cohomologically affine, it is a good moduli space morphism, so it is universal for maps to algebraic spaces 
  \cite[Theorem 6.6]{Alper:good}\footnote{Technically, this is not a direct application of \cite[Theorem 6.6]{Alper:good}, as Alper's result assumes $\Y$ is an algebraic space. However, since algebraic spaces are sheaves in the smooth topology, this property may be checked smooth locally on $\Y$. Good moduli space morphisms are stable under base change by \cite[Proposition 4.7(i)]{Alper:good}, so we are reduced to the case when $\Y$ is an algebraic space.}, so it is a relative coarse space morphism. Since $\X$ is smooth and tame, it follows that $\Y$ has tame quotient singularities. Since $g$ is an isomorphism away from codimension 2, it is a canonical stack morphism by Remark \ref{rmk:canonical-universal-property}. This completes the proof that $\pi$ factors as $\X\cong \sqrt{\D/X^\can}^\can\to \sqrt{\D/X^\can}\to X^\can\to X$.

If $D$ is Cartier, then the above proof may be modified by replacing $X^\can$ by $X$, $\D$ by $D$, and removing the singular locus of $X$ from $\U$. The modified proof shows that $\pi$ factors as $\X\cong \sqrt{D/X}^\can\to \sqrt{D/X}\to X$.

\section*{A local description of Theorem \ref{thm:bottom-up-smooth-DM}}

In this section, we provide a representation-theoretic description of the factorization of the coarse space morphism for stacks of the form $\X=[V/G]$, where $V$ is a vector space over $k$ on which an abstract finite group $G$ acts linearly. Since every separated Deligne-Mumford stack is \'etale locally a quotient by the stabilizer of a point \cite[Lemma 2.2.3 and its proof]{Abramovich-Vistoli} and smooth $k$-varieties with group actions have equivariant \'etale morphisms to their tangent spaces, this approach actually provides an alternative proof of the theorem, but we omit those details here.

\begin{theorem}\label{thm:local-characterization}
  Suppose $\X=[V/G]$, where $V$ is a vector space over $k$ on which an abstract finite group $G$ acts linearly and faithfully. Let $H\subseteq G$ be the subgroup generated by pseudoreflections of $V$, and $H'\subseteq H$ be its commutator subgroup. Then the factorization of the coarse space morphism in Theorem \ref{thm:bottom-up-smooth-DM} is
\[\xymatrix @R-1.2pc{
  \llap{$\X\cong\;$} \sqrt{\D/X^\can}^\can \ar[r]\ar@{}[d]|\parallel & \sqrt{\D/X^\can} \ar[r]\ar@{}[d]|\parallel & X^\can \ar[r]\ar@{}[d]|\parallel & X\ar@{}[d]|\parallel\\
  [V/G] \ar[r] & [(V/H')/(G/H')] \ar[r] & [(V/H)/(G/H)] \ar[r] & V/G
}\]
\end{theorem}
\begin{proof}
  Note that for a finite group $K$ acting linearly on a vector space $U$, the coarse space morphism $[U/K]\to U/K$ is a canonical stack morphism if and only if $K$ acts without pseudoreflections.

  By the Chevalley-Shephard-Todd theorem, $V/H$ is a vector space. 
  Moreover, $G/H$ acts linearly on $V/H$ without pseudoreflections. Therefore $[(V/H)/(G/H)]\to V/G$ is a canonical stack morphism. This is the usual construction of the canonical stack.

  The action of $H'$ on $V$ has no pseudoreflections since commutators act with determinant $1$, and pseudoreflections do not. Therefore $[V/H']\to V/H'$ is a canonical stack morphism. Since $(V/H')\to [(V/H')/(G/H')]$ is an \'etale cover, and the property of being a canonical stack morphism can be checked \'etale locally on the target, $[V/G]\to [(V/H')/(G/H')]$ is also a canonical stack morphism.

  It remains to show that $[(V/H')/(G/H')]\to [(V/H)/(G/H)]$ is a root stack morphism. Since this is a property which can be checked \'etale locally on the target, it suffices to show that $[(V/H')/(H/H')]\to V/H$ is a root stack morphism.

  For any pseudoreflection hyperplane $W$, let $H\cdot W$ be the set of $H$-translates of $W$. For each hyperplane $U$ in this orbit, choose a linear function $\ell_U$ which vanishes on $U$, and let $f_{H\cdot W}=\prod_{U\in H\cdot W} \ell_U$. The vanishing locus of $f_{H\cdot W}$ is $H$-invariant, so $H$ acts on $f_{H\cdot W}$ by some character $\chi_{H\cdot W}$ of $H$.

  Given a pseudoreflection $r\in H$, we can break $H\cdot W$ up into $r$-orbits. We analyze the contribution of each of these cycles to $\chi_{H\cdot W}(r)$. If a hyperplane $U$ in $H\cdot W$ is fixed by $r$, then either $r$ is a pseudoreflection through $U$ (in which case $r(\ell_U)=\det(r)^{-1}\cdot \ell_U$) or $r$ is a pseudoreflection through a plane perpendicular to $U$ (in which case $r(\ell_U)=\ell_U$). Now consider the case where $U$ is not fixed by $r$, letting $U_i=r^i(U)$. Since $U$ is not fixed by $r$, $\ell_U$ has non-zero components in both $W$ and $W^\perp$, so $\{U_0, U_1, \dots, U_{\ord(r)-1}\}$ are distinct hyperplanes. Since $r(\ell_{U_i})$ is a linear function vanishing on $U_{i+1}$, we have that $r(\ell_{U_i})=a_i\cdot\ell_{U_{i+1}}$ for some constant $a_i$. We have that $r$ acts on $\prod_{i=0}^{\ord(r)-1} \ell_{U_i}$ by $\prod_{i=0}^{\ord(r)-1}a_i$. On the other hand, $\ell_{U_{\ord(r)}}=\ell_{U_0} = r^{\ord(r)}(\ell_{U_0}) = (\prod_{i=0}^{\ord(r)-1}a_i)\cdot \ell_{U_{\ord(r)}}$, so $\prod_{i=0}^{\ord(r)-1}a_i = 1$. Summing up, if $r$ is a pseudoreflection through a hyperplane in $H\cdot W$, then $\chi_{H\cdot W}(r) = \det(r)^{-1}$; otherwise $\chi_{H\cdot W}(r) = 1$.

  Let $W_1$, \dots, $W_k$ be $H$-orbit representatives (so every pseudoreflection hyperplane is in the $H$-orbit of exactly one of the $W_i$), and $r_1$, \dots, $r_k$ be primitive pseudoreflections through the $W_i$, of orders $e_1$, \dots, $e_k$, respectively. We then have that the map $H/H'\to \mu_{e_1}\times\cdots\times \mu_{e_k}$ induced by $\chi_{H\cdot W_1}\times\cdots\times \chi_{H\cdot W_k}$ is an isomorphism.

Let $V/H = \spec(k[g_1, \dots, g_n])$. Then we claim that $V/H' = \spec(k[g_1,\dots, g_n, f_{H\cdot W_1}, \dots, f_{H\cdot W_k}])$. Note first that each $f_{H\cdot W_i}$ is $H'$-invariant, since $H$ acts on it by the character $\chi_{H\cdot W_i}$, and $H'$ is in the kernel of any character. To see that this is the full ring of $H'$-invariants, note that any $H'$-invariant function is a sum of $H/H'$-semi-invariant functions, so it suffices to show that any $H/H'$-semi-invariant function $f$ is in this ring. If each $r_i$ acts trivially on $f$, then $f$ is $H$-invariant, so it is some polynomial combination of the $g_j$. If $r_i$ acts non-trivially on $f$, then $f$ must vanish along $W_i$. Since $f$ is $H'$-invariant, all the conjugates of $r_i$ must also act non-trivially on $f$, so $f$ must vanish along all of $H\cdot W_i$, so it must be divisible by $f_{H\cdot W_i}$, and $f/f_{H\cdot W_i}$ is $H'$-invariant.  By induction on the degree of $f$, it is a polynomial in the $g_j$ and $f_{H\cdot W_i}$.

We therefore have an $H/H'$-equivariant closed immersion of $V/H'$ into $\AA^{n+k}$, where $\mu_{e_i}$ acts only on the $(n+i)$-th coordinate. The following diagram is cartesian:
\[\xymatrix{
  \bigl[(V/H')\bigm/(H/H')\bigr] \ar[r]\ar[d] & \bigl[\AA^{n+k}\bigm/(H/H')\bigr]\ar[d]\\
  V/H \ar[r] & \AA^{n+k}/(H/H')
}\]

Since each $\mu_{e_i}$ acts on only one coordinate of $\AA^{n+k}$, the right vertical arrow is a root stack morphism ramified to order $e_i$ along the $(n+i)$-th coordinate hyperplane. It follows that the left arrow is a root stack morphism ramified to order $e_i$ along the divisor which is the image of $H\cdot W_i$.
\end{proof}

\begin{remark}\label{rmk:reflections->ramification}
The proof of Theorem \ref{thm:local-characterization} shows that for a quotient stack $\X=[U/G]$, the ramification divisor of the map $[U/G]\to U/G$ is the image of the pseudoreflection divisors, and the ramification degrees are the orders of the pseudoreflections.
\end{remark}

\section*{Examples, Counterexamples, and Questions}

We begin this section by showing that the casual description of a stack $\X$ as a variety $X$ together with stabilizer groups attached to some of its subvarieties does not uniquely specify $\X$.

\begin{example}[Different stacks with the same coarse space and stabilizers]\label{eg:non-affine-lines}
We work over a field $k$ with $char(k)\neq2$.
\begin{enumerate}
\item (\emph{Smooth separated stack}) Consider the action of $\ZZ/2$ on $\AA^1$ given by $x\mapsto -x$. The quotient $\X=[\AA^1/(\ZZ/2)]$ is a smooth separated Deligne-Mumford stack with trivial generic stabilizer. Its coarse space is $X=\AA^1$ (with coordinate $x^2$) and it has a single $\ZZ/2$ stabilizer at the origin.

\item (\emph{Singular separated stack}) Consider the singular Deligne-Mumford stack $\X$ given by the quotient of the axes in $\AA^2$ by the $\ZZ/2$-action $(x,y)\mapsto(y,x)$. The coarse space of $\X$ is also $\AA^1$ (with coordinate $x+y$) and the stabilizer at the origin is $\ZZ/2$.

\item (\emph{Smooth non-separated stack}) Let $G$ be the non-separated affine line. We can think of $G$ as a group scheme over $\AA^1$ via the natural map $G\to \AA^1$. The stack $BG$ has coarse space $\AA^1$ and a stabilizer of $\ZZ/2$ at the origin. Note that $BG$ is a smooth Deligne-Mumford stack with trivial generic stabilizer, but is not separated (its diagonal isn't separated, so not proper). Note also that the map $BG\to \AA^1$ is a coarse space morphism (i.e.~it is a universal map from $BG$ to an algebraic space, and is bijective on geometric points), but not a good moduli space morphism since it is not cohomologically affine.\qedhere
\end{enumerate}
\end{example}

In the next example, we illustrate Theorem \ref{thm:bottom-up-smooth-DM} with $\X=[\AA^n/S_n]$.

\begin{example}
\label{ex:An-mod-Sn}
Consider the stack $\X=[\AA^n/S_n]$, where the $S_n$-action permutes the coordinates, over a field whose characteristic is prime to $n!$. The coarse space is $X=\AA^n$, with coordinates given by the elementary symmetric functions. The $S_n$-action is generated by pseudoreflections corresponding to transpositions; since all transpositions are in the same conjugacy class as $(1\ 2)$, the ramification locus of the coarse space map $\X\to X$ is given by $(x_1-x_2)\cap k[x_1,\dots, x_n]^{S_n}$, which is generated by the descriminant $\Delta := \prod_{i-j}(x_i-x_j)^2$, and the order of ramification is $2$ (see Remark \ref{rmk:reflections->ramification}). The ramification divisor $D\subseteq X$ is given by expressing $\Delta$ in terms of the elementary symmetric functions. Up to sign, this is given by the resultant of $f$ and $f'$, where $f(x)=x^n-e_1x^{n-1}+\cdots + (-1)^n e_n$ is the polynomial with the elementary symmetric functions as coefficients.

Since the action of $S_n$ on $\AA^n$ is the sum of the trivial and standard representations of $S_n$, the features of this example are all present if one restricts to the vanishing locus of $e_1=x_1+\cdots +x_n$, the first elementary symmtric function. In this case, the coarse space of $[\AA^{n-1}/S_n]$ is $\AA^{n-1}=\spec k[e_2,\dots, e_n]$, the ramification divisor is given by the discriminant of $f(x)=x^n+e_2x^{n-2}+\cdots + (-1)^ne_n$, and the ramification order is 2.
\end{example}

\begin{example}[Square root of the cuspidal cubic]
In this example, we construct the smooth stack $\X$ with coarse space $\AA^2=\spec \CC[a,b]$ so that the ramification divisor is the cuspidal cubic $D=V(a^3-b^2)$ with ramification order $2$. The coarse space is already smooth, so it is equal to its canonical stack. The square root of the cuspidal cubic is $\sqrt{D/\AA^2} = [Y/(\ZZ/2)]$, where $Y=\spec \CC[a,b,t]/(b^3-a^2-t^2)$ is the $A_2$-singularity, and $\ZZ/2$ acts as $t\mapsto -t$. The canonical stack of $Y$ is $Y^\can=[(\spec \CC[x,y])/(\ZZ/3)]$ where the $\ZZ/3$-action is given by $(x,y)\mapsto (\zeta x,\zeta^2 y)$ with $\zeta$ a third root of unity. The relation between the $x,y$ coordinates and the $a,b,t$ coordinates is $x^3=a+ti$, $y^3=a-ti$, and $xy=b$. We see that the $\ZZ/2$-action on $Y$ fixes $xy$ and swaps $x^3$ and $y^3$. It therefore lifts to an action on $\AA^2$ swapping $x$ and $y$. Together with the $\ZZ/3$-action, this lifted $\ZZ/2$-action generates an action of $S_3$ on $\AA^2$. This $S_3$-action is the unique $2$-dimensional irreducible representation. We therefore see that $\X = [Y/(\ZZ/2)]^\can = [Y^\can/(\ZZ/2)] = [\AA^2/S_3]$. This is the $n=3$ case of Example \ref{ex:An-mod-Sn}.
\end{example}
\begin{remark}
The above example can be done over $\RR$ instead of over $\CC$, but since $\RR$ has no primitive cube root of unity, the stabilizer of the origin is $(\ZZ/2)\ltimes \mu_3$, a twisted version of $S_3$. To get a stabilizer of $S_3$, one would square root the cuspidal cubic given by the discriminant $\Delta = 4a^3-27b^2$.
\end{remark}

\begin{example}[Reducible ramification divisor without simple normal crossings]
Let $k$ be a field with primitive 4-th roots of unity and with $\textrm{char}(k)\neq2$. Consider $D\subseteq \AA^2$ defined by $y(x^2-y)=0$. Then $\sqrt{D/\AA^2} = [Y/(\ZZ/2)]$ where $Y=\spec k[x,y,t]/(y^2-x^2y+t^2)$ and $\ZZ/2$ acts as $t\mapsto -t$. Letting $z=y-\frac{1}{2}x^2$, we see that the defining equation of $Y$ is $\frac{1}{4}x^4=z^2+t^2$. After a further change of coordinates $u=2(z+ti)$ and $v=2(z-ti)$, we see that $Y$ is the $A_3$-singularity defined by the equation $x^4=uv$. In these coordinates, the $\ZZ/2$-action on $Y$ fixes $x$ and swaps $u$ and $v$. Since $Y$ is the $A_3$-singularity, $Y^\can=[\AA^2/(\ZZ/4)]$ where $\ZZ/4$ acts on $\AA^2$ by $(a,b)\mapsto (ia,-ib)$. The relation between the $a,b$ coordinates and the $x,u,v$ coordinates is given by $x=ab$, $u=a^4$, and $v=b^4$. Note that the $\ZZ/2$-action on $Y$ lifts to $\AA^2$ by swapping $a$ and $b$. With the action of $\ZZ/4$, this generates an action of the dihedral group $D_8$, and so $\sqrt{D/\AA^2}^\can = [\AA^2/D_8]$. This argument generalizes to show that $\sqrt{V(y(x^n-y))/\AA^2}^\can = [\AA^2/D_{4n}]$.
\end{example}

\begin{example}[Singular coarse space]
Let $X=\spec k[a,b,c]/(ab-c^m)$ be the $A_{m-1}$-singularity and let $D\subseteq X$ be a Weil divisor which generates the class group of $X$. The canonical stack is $X^\can=[\spec(k[x,y])/\mu_m]$ where the action is $\zeta\cdot (x,y)=(\zeta x,\zeta^{-1}y)$, with $(a,b,c)=(x^m,y^m,xy)$. $D$ lifts to the divisor $[V(x)/\mu_m]$. The root stack $\sqrt[n]{\D/X^\can}$ is then 
\[
  \sqrt[n]{[V(x)/\mu_m]\bigm/[\spec k[x,y]/\mu_m]} = [\sqrt[n]{V(x)/\spec k[x,y]}/\mu_m] = \bigl[[\spec(k[x,y,t]/(x-t^n))/\mu_n]/\mu_m\bigr]
\]
where $\mu_n$ acts as $\zeta\cdot (x,y,t)=(x,y,\zeta t)$. This stack is the quotient $[\spec(k[y,t])/\mu_{mn}]$, where the action is $\zeta\cdot (y,t)=(\zeta^{-1}y,\zeta^m t)$.
\end{example}

Given a smooth separated Deligne-Mumford stack $\X$ with trivial generic stabilizer, we can associate to it the tuple $(X,D_i,e_i)$, where $X$ is its coarse space, the $D_i$ are the irreducible components of the ramification divisor of the coarse space map $\pi:\X\to X$, and $e_i$ is ramification degree of $\pi$ over $D_i$. Although our main result shows that $\X$ is characterized by $(X,D_i,e_i)$, we now show that not every tuple $(X,D_i,e_i)$ arises in this manner.

\begin{example}[Not every tuple $(X,D_i,e_i)$ arises]\label{ex:not-every-pair}
Consider the cone $D=V(xy+z^2)$ in $X=\AA^3$. Taking the square root of $X$ along $D$ yields a stack $\Y$ with a singularity of the form $xy+z^2=t^2$. This is isomorphic to the toric singularity $xy=uv$ (with $(u,v)=(t+z,t-z)$), which is not a quotient singularity (and hence has no canonical Deligne-Mumford stack). It follows from Theorem \ref{thm:bottom-up-smooth-DM} that there is no smooth separated Deligne-Mumford stack $\X$ with trivial generic stabilizer and coarse space $X$ whose coarse space morphism $\X\to X$ is ramified over $D$ with degree 2.
\end{example}

In light of this example, we ask:

\begin{question}
\label{q:which-X-D-occur}
Which tuples $(X,D_i,e_i)$ arise in the manner described above? That is, given an algebraic space $X$ with tame quotient singularities, a reduced Weil divisor $D$ on $X$ with irreducible components $D_i$, and integers $e_i\geq 2$, is there a smooth separated Deligne-Mumford stack $\X$ with a coarse space morphism $\X\to X$ ramified over $D_i$ with degree $e_i$?
\end{question}

\begin{remark}
It suffices to answer Question \ref{q:which-X-D-occur} formally locally: which singularity types and ramification orders can arise for the quotient of a vector space $V$ by a faithful linear action a finite group $G$? By Remark \ref{rmk:reflections->ramification}, the ramification divisors are the images of the hyperplanes of pseudoreflections, and the ramification degrees are the orders of the pseudoreflections.
\end{remark}

\begin{example}[Even if $(X,D_i,e_i)$ arises, $(X,D_i,ne_i)$ may not]
 Let $D\subseteq \AA^2_\CC$ be the cuspidal cubic $x^3-y^2=0$. The root stack $\sqrt[n]{D/X}$ has a unique singularity, which is of the form $x^3-y^2=z^n$. For $n=2,3,4,5$, this is a singularity of type $A_2$, $D_4$, $E_6$, and $E_8$, respectively (see e.g.~\cite[\S4.2]{reid_algebraic_surfaces} for descriptions of these singularities).\footnote{It is immediate that $x^3-y^2=z^n$ is the $A_2$, $E_6$, and $E_8$ singularity for $n=2, 4, 5$, respectively. To see that $x^3-y^2=z^3$ is the $D_4$-singularity, make the coordinate change $(u,v) = (\sqrt[3]{4/7}(x+z),\sqrt[6]{4/189}\frac{1}{2}(z-x))$.} In general, for $n\geq 6$, the singularity $x^3-y^2=z^n$ is not a quotient singularity; for $6\le n\le 11$, this singularity is elliptic \cite[Exercise 18, Chapter 4]{reid_algebraic_surfaces}. Thus, $(\AA^2,D,n)$ arises as the coarse space and ramification divisor of a smooth separated Deligne-Mumford stack for $n\leq 5$, but in general not for $n\ge 6$.
\end{example}

In light of this example, we ask:

\begin{question}
  If $X$ has tame quotient singularities and $D\subseteq X$ is an irreducible divisor which does not have simple normal crossings, does $(X,D,n)$ fail to arise for all sufficiently large $n$?
\end{question}


\end{document}